\documentclass{article}

\usepackage{makeidx}
\usepackage{graphicx}
\usepackage{multicol}
\usepackage[bottom]{footmisc}
\usepackage{amsmath,amssymb,amsfonts, amsthm}
\usepackage[affil-it]{authblk}

\makeindex

\newtheorem{thm}{Theorem}
\newtheorem{lem}{Lemma}
\newtheorem{cor}{Corollary}
\newtheorem{obs}{Observation}
\theoremstyle{definition}
\newtheorem{defn}{Definition}
\newtheorem*{lem'}{Lemma}

\theoremstyle{remark}
\newtheorem{case}{Case}[thm]
\newtheorem{rem}{Remark}[section]

\author{Santanu Mondal, Krishnendu Paul, Shameek Paul
\thanks{E-mail addresses:   \texttt{santanu.mondal.math18@gm.rkmvu.ac.in, krishnendu.p.math18@gm.rkmvu.ac.in, shameek.paul@rkmvu.ac.in}}}

\affil{\small Ramakrishna Mission Vivekananda Educational and Research Institute, Belur, Dist. Howrah, 711202, India}

\date{}

\begin{document}
\baselineskip=14.5pt

\title {On a different weighted zero-sum constant}

\maketitle

\begin{abstract}
For a finite abelian group $(G,+)$, the constant $C(G)$ is defined to be the smallest natural number $k$ such that any sequence in $G$ having length $k$ will have a subsequence of consecutive terms whose sum is zero. For a subset $A\subseteq\mathbb Z_n$, the constant $C_A(n)$ is the smallest natural number $k$ such that any sequence in $G$ having length $k$ has an $A$-weighted zero-sum subsequence of consecutive terms. We determine the value of $C_A(n)$ for some particular weight-sets $A$.
\end{abstract}

\bigskip


\vspace{.7cm}

\section{Introduction}\label{0}

For a finite set $S$, we denote the number of elements in $S$ by $|S|$. For $a,b\in\mathbb Z$, let $[a,b]$ denote the set $\{k\in\mathbb Z:a\leq k\leq b\}$.  We begin with the following well-known result (see \cite{AR1}, for instance).

\begin{thm}\label{a}
Let $(G\,,\cdot)$ be a finite group with $|G|=n$ and $k\geq n$. Then given any sequence $S=(x_1,\ldots,x_k)$ in $G$ of length $k$, there exist $i,j\in [1,k]$ such that $i\leq j$ and $x_ix_{i+1}\ldots x_j=e$ where $e$ is the identity element of $G$. 
\end{thm}

\begin{proof}
 Let $S=(x_1,\ldots,x_k)$ be a sequence in $G$ and $y_i=x_1x_2\ldots x_i$ for each $i\in [1,k]$. If some $y_i=e$, we are done. Else by the pigeonhole principle, there exist $i,j\in [1,k]$ such that $i<j$ and $e=y_i^{-1}y_j=x_{i+1}x_{i+2}\ldots x_j$. 
\end{proof}

\begin{defn}
A sequence $S=(x_1,\ldots,x_\ell)$ in $G$ is called a product-identity sequence if $x_1x_2\ldots x_\ell=e$.
\end{defn}

\begin{defn}
For a finite group $G$, the {\it Davenport constant} $D(G)$ is defined to be the smallest natural number $k$ such that any sequence of $k$ elements in $G$ has a product-identity subsequence. 
\end{defn}

From Theorem \ref{a}, we see that for a finite group $G$ we have $D(G)\leq |G|$. A weighted generalization of the Davenport constant was introduced in \cite{AC} for finite abelian groups. It was earlier introduced in \cite{AR} for finite cyclic groups, following a similar generalization in \cite{ACF}. We give a generalization of the weighted Davenport constant to finite $R$-modules. 
For the rest of this section, $R$ will be a ring with unity and $A$ will be a non-empty subset of $R$.

\begin{defn}
Given an $R$-module $M$ and $A\subseteq R$, a sequence $S=(x_1,\ldots,x_k)$ in $M$ is called an {\it $A$-weighted zero-sum sequence} if for each $i\in [1,k]$, there exists $a_i\in A$ such that $a_1x_1+\cdots+a_kx_k=0$. When $A=\{1\}$, an $A$-weighted zero-sum sequence is also called a zero-sum sequence.
\end{defn}

\begin{defn} 
For a finite $R$-module $M$ and $A\subseteq R$, the {\it $A$-weighted Davenport
constant of $M$} denoted by $D_A(M)$ is defined
to be the least positive integer $k$ such that any sequence in $M$ of length $k$ has an $A$-weighted zero-sum subsequence. 
\end{defn}

From Theorem \ref{a}, we notice that for a finite abelian group $(G,+)$, we actually get a zero-sum subsequence which consists of {\it consecutive} terms of the given sequence. This motivates the following definition.

\begin{defn}
For a finite $R$-module $M$ and $A\subseteq R$, we define the constant $C_A(M)$ to be the least positive integer $k$ such that any sequence in $M$ of length $k$ has an $A$-weighted zero-sum subsequence of consecutive terms. 
\end{defn}

\begin{obs}
For a finite $R$-module $M$ and $A\subseteq R$, we claim that both the constants $D_A(M)$ and $C_A(M)$ exist. Let $M$ be a finite $R$-module. Given any sequence $S$ in $M$ of length $|M|$, by a similar argument as in the proof of Theorem \ref{a} we see that we can find a zero-sum subsequence $T$ of $S$ which has consecutive terms. By multiplying the zero-sum by an element $a\in A$, we see that $T$ is an $A$-weighted zero-sum subsequence of $S$. Hence, for any $A\subseteq R$ we have $C_A(M)\leq |M|$. Also, for any $A\subseteq R$ we clearly have $D_A(M)\leq C_A(M)$. 
\end{obs}

When $A=\{1\}$, we denote the constants  $C_A(M)$ and $D_A(M)$ by $C(M)$ and $D(M)$ respectively.   We consider the ring $\mathbb Z_n$ as a $\mathbb Z_n$-module and for $A\subseteq \mathbb Z_n$, we denote the constants $C_A(\mathbb Z_n)$ and $D_A(\mathbb Z_n)$ by $C_A(n)$ and $D_A(n)$  respectively. The next result is an immediate consequence of Theorem \ref{a}. Here we consider an abelian group $G$ as a $\mathbb Z$-module. 

\begin{cor}\label{cn}
For a cyclic group $G$ we have $C(G)=|G|$.
\end{cor}

\begin{proof}
For any group $G$ we have $D(G)\leq C(G)\leq |G|$. Let $G=\langle a\rangle$ be a cyclic group of order $n\geq 2$. By considering the constant sequence $(a,\ldots,a)$ of length $n-1$, we see that $D(G)\geq n$. Hence, for a cyclic group $G$ we have $C(G)=|G|$. 
\end{proof}

\begin{thm}\label{nz}
For $A=\mathbb Z_n\setminus \{0\}$ we have $C_A(n)=2$. 
\end{thm}

\begin{proof}
For any $A\subseteq \mathbb Z_n\setminus \{0\}$ we have seen that $C_A(n)\geq 2$. Let $S=(x_1,x_2)$ be a sequence in $\mathbb Z_n$. We claim that $S$ has an $A$-weighted zero-sum subsequence of consecutive terms. If either $x_1$ or $x_2$ is zero, we get a zero-sum subsequence of length 1. If both $x_1$ and $x_2$ are non-zero, then $a_1=x_2,~a_2=-x_1\in A$ and we have $a_1x_1+a_2x_2=0$. This shows that our claim is true and hence $C_A(n)\leq 2$. Thus $C_A(n)=2$.
\end{proof} 

Let $U(n)$ denote the multiplicative group of units in the ring $\mathbb Z_n$. If $p$ is a prime, by Theorem \ref{nz} it follows that $C_{U(p)}(p)=2$. For $j\geq 1$ let $U(n)^j$ denote the set $\{\,x^j : x\in U(n)\,\}$. For $n=p_1p_2\ldots p_k$ where $p_i$ is a prime for each $i\in [1,k]$, we define $\Omega(n)=k$. For a divisor $m$ of $n$, we define the homomorphism $f_{n,m}:\mathbb Z_n\to \mathbb Z_m$ as  $f_{n,m}(a+n\mathbb Z)=a+m\mathbb Z$. In this article the following are among some of the results which we have obtained.

\begin{itemize}
 \item For any odd natural number $n$, we have $C_{U(n)}(n)=2^{\Omega(n)}$.
 
 \item For any prime $p$, we have $C_{U(p)^2}(p)=3$ when $p\neq 2$ and $C_{U(2)^2}(2)=2$. 
 
 \item If every prime divisor of $n$ is at least 7, we have $C_{U(n)^2}(n)=3^{\Omega(n)}$. 
 
 \item If $p$ is a prime such that $p\equiv 1~(mod~3)$, we have $C_{U(p)^3}(p)=D_{U(p)^3}(p)=3$ when $p\neq 7$. Also we have $C_{U(7)^3}(7)=4$ and $D_{U(7)^3}(7)=3$. 
 
 \item For a squarefree number $n$, we have $C_{U(n)^3}(n)=2^{\Omega(n_2)}3^{\Omega(n_1)}$ if $n$ is not divisible by 2,7 or 13. (The notation ``$n=n_1n_2$'' is defined in Section \ref{cubes}.)
 
 \item For any number $n$, we have $2^{\Omega(n_2)}3^{\Omega(n_1)}\leq C_{U(n)^3}(n)\leq 2^{\Omega(n_2)}4^{\Omega(n_1)}$ if $n$ is not divisible by 2, 3 or 7. (The notation is as in the previous result.)
 
 \item Let $n=m_1m_2$ and $A,A_1,A_2$ be subsets of $\mathbb Z_n,\mathbb Z_{m_1},\mathbb Z_{m_2}$ respectively. If $f_{n,m_1}(A)\subseteq A_1$ and $f_{n,m_2}(A)\subseteq A_2$, then $C_A(n)\geq C_{A_1}(m_1)C_{A_2}(m_2)$. 
\end{itemize}

\begin{obs} Let $A\subseteq \mathbb Z_n\setminus \{0\}$ and let $m$ be given. Consider the sequence $$(1,\underbrace{0,\ldots,0}_{m-1},1,\underbrace{0,\ldots,0}_{m-1},1,0,\ldots)$$ in $\mathbb Z_n$ (of arbitrary length). This sequence does not have any $A$-weighted zero-sum subsequence of consecutive terms of length $m$. This shows that we cannot have a similar definition like that of $C_A(G)$ in which we place a restriction on the length of the $A$-weighted zero-sum subsequence.
\end{obs}

\section{When $A=U(p)^2$ where $p$ is a prime}\label{q}

The next result follows from the well-known Cauchy-Davenport theorem (\cite{Nat}, Theorem 2.3). 

\begin{thm}\label{cd}
Let $p$ be a prime and $X,~Y,~W$ be subsets of $\mathbb Z_p$. Then either $X+Y+W=\mathbb Z_p$ or $|X+Y+W|\geq |X|+|Y|+|W|-2$.  
\end{thm}

For an odd prime $p$ we denote $U(p)^2$ by $Q_p$. As $Q_p$ is the image of the homomorphism $U(p)\to U(p)$ given by $x\mapsto x^2$ whose kernel is $\{1,-1\}$, it follows that $|Q_p|=(p-1)/2$. We denote $U(p)\setminus Q_p$ by $N_p$.

For an odd prime $p$, in \cite{AR} it was shown that when $A=Q_p$ (or $N_p$) we have  $D_A(p)=3$. By a similar argument, the next result is also true when $A=N_p$. 

\begin{thm}\label{q1}
For a prime $p$, we have $C_{Q_p}(p)=3$ when $p\neq 2$ and $C_{Q_2}(2)=2$. 
\end{thm}

\begin{proof}
When $p=2$ or $3$ we have $Q_p=\{1\}$ and so $C_{Q_p}(p)=C(p)$. Hence by Corollary \ref{cn}, we have $C_{Q_2}(2)=2$ and $C_{Q_3}(3)=3$. When $p=5$ we have $Q_5=\{1,-1\}$. Suppose $S=(x,y,z)$ is a sequence in $\mathbb Z_5$. Consider the set $\{x\pm y,-x\pm y,\pm z\}$ which has six elements of $\mathbb Z_5$. As at least two elements from this set are equal, we get a $Q_5$-weighted zero-sum subsequence of consecutive terms of $S$. Hence $C_{Q_5}(5)\leq 3$.

For a prime $p\geq 7$ let $S=(x,y,z)$ be a sequence in $\mathbb Z_p$. If some term of $S$ is zero, we get a $Q_p$-weighted zero-sum subsequence of length one. Suppose all the terms of $S$ are non-zero. For $w\in \mathbb Z_p$ let $Q_pw=\{aw:a\in Q_p\}$. It follows that $|Q_px|=|Q_py|=|Q_pz|=|Q_p|=(p-1)/2$. So $|Q_px|+|Q_py|+|Q_pz|-2=3(p-1)/2-2=(3p-7)/2$. When $p\geq 7$ we have $(3p-7)/2\geq p$. Thus, by Theorem \ref{cd} we get that $Q_px+Q_py+Q_pz=\mathbb Z_p$ and so $S$ is a $Q_p$-weighted zero-sum sequence. Hence, when $p\geq 7$ we have $C_{Q_p}(p)\leq 3$.

If $x\in N_p$ for a prime $p\geq 5$, we see that $(-1,x)$ is a sequence in $\mathbb Z_p$ which does not have any $Q_p$-weighted zero-sum subsequence. Hence, when $p\geq 5$ we have $C_{Q_p}(p)\geq 3$. Thus, from all the above results we get $C_{Q_p}(p)=3$ for a prime $p\geq 3$. 
\end{proof}

\section{When $A=U(p)^3$ where $p$ is a prime}

When $p\not\equiv 1~(mod~3)$, there is no element of order three in $U(p)$. So the kernel of $\varphi:U(p)\to U(p)$ given by $x\mapsto x^3$ is trivial and hence $U(p)^3=U(p)$. In this case, we have seen that $C_{U(p)}(p)=2$.

When $p\equiv 1~(mod~3)$, there is an element $c$ which has order three in $U(p)$. So the kernel of $\varphi$ is the cyclic subgroup generated by $c$. As the image of $\varphi$ is $U(p)^3$, it follows that $U(p)^3$ is a subgroup of index 3 in $U(p)$.

We will use the following results which are the first Theorem and Proposition 6.1 from \cite{LS}.

\begin{thm}\label{ls1}
 Let $F$ be a field with $|F|\neq 4,7,13,16$. Suppose $G$ is a subgroup of index $3$ in $F^*$. Then we have $G+G=F$. 
\end{thm}

\begin{thm}\label{ls2}
 Let $F$ be a finite field with $|F|\neq 4,7$. Suppose $G$ is a subgroup of index 3 in $F^*$.  If $a\in G+G$ with $a\notin G\cup\{0\}$, then we have $G+aG=F^*$. 
\end{thm}

\begin{lem}\label{cub5}
Let $p$ be a prime such that $p\equiv 1~(mod~3)$ and $p\neq 7,13$. Suppose $S$ is a sequence in $\mathbb Z_p$ such that at least three terms of $S$ are in $U(p)$. Then $S$ is a $U(p)^3$-weighted zero-sum sequence.  
\end{lem}

\begin{proof}
Let $S$ be a sequence in $\mathbb Z_p$ and $x,y,z$ be terms of $S$ which are units. Let $w$ be the sum of the remaining terms of $S$ (if any). The equation $zX^3=w$ has at most three roots in $\mathbb Z_p$. As there are at least four elements in $U(p)$ when $p>5$, we can find $t\in U(p)$ such that $zt^3\neq w$. So if $z'=w-zt^3$, then we have $z'\neq 0$. To prove that $S$ is a $U(p)^3$-weighted zero-sum sequence, it is enough to show that the sequence $S'=(x,y,z')$ is a $U(p)^3$-weighted zero-sum sequence as we have $-t^3\in U(p)^3$.

For any $c\in U(p)$, the sequence $(cx,cy,cz')$ is a $U(p)^3$-weighted zero-sum sequence if and only if the sequence $S'$ is a $U(p)^3$-weighted zero-sum sequence. So we can assume that $x\in U(p)^3$. From Theorems \ref{ls1} and \ref{ls2} (depending on whether $y\in U(p)^3$ or $y\notin U(p)^3$), we see that $-z'\in U(p)^3+U(p)^3y$. As $x\in U(p)^3$, we see that $U(p)^3=U(p)^3x$. Thus, there exist $a\in U(p)^3$ and $b\in U(p)^3$ such that $-z'=ax+by$. Hence, $S'$ is a $U(p)^3$-weighted zero-sum sequence. 
\end{proof}

\begin{rem}
The conclusion of Lemma \ref{cub5} is not true when $p=7$ and $13$. This is because we can check that the sequence $(1,1,1)$ in $\mathbb Z_p$ is not a $U(p)^3$-weighted zero-sum sequence when $p=7$ and $13$ as $U(7)^3=\{\pm 1\}$ and $U(13)^3=\{\pm 1,\pm 5\}$. 
\end{rem}

\begin{thm}\label{cubp13}
If $p$ is a prime such that $p\equiv 1~(mod~3)$, we have $D_{U(p)^3}(p)\geq 3$. If in addition $p\neq 7$, we have $D_{U(p)^3}(p)=C_{U(p)^3}(p)=3$. 
\end{thm}

\begin{proof}
Let $x\in U(p)\setminus U(p)^3$ where $p$ is a prime such that $p\equiv 1~(mod~3)$. As the sequence $(-1,x)$ does not have any $U(p)^3$-weighted zero-sum subsequence, it follows that  $D_{U(p)^3}(p)\geq 3$. 

Let $S=(x,y,z)$ be a sequence in $\mathbb Z_p$ where $p\neq 7,13$. We want to show that $S$ has a $U(p)^3$-weighted zero-sum subsequence of consecutive terms. We can assume that $x,y,z\in U(p)$. By Lemma \ref{cub5} we see that $S$ is a $U(p)^3$-weighted zero-sum sequence. Hence $C_{U(p)^3}(p)\leq 3$. As $D_{U(p)^3}(p)\leq C_{U(p)^3}(p)$, we get that $D_{U(p)^3}(p)=C_{U(p)^3}(p)= 3$.

As $C_{U(13)^3}(13)\geq D_{U(13)^3}(13)\geq 3$, it now remains to show that $C_{U(13)^3}(13)\leq 3$.  
Let $S=(x,y,z)$ be a sequence in $\mathbb Z_{13}$. We may assume that $x,y,z\in U(13)$. By multiplying the terms of $S$ by an element of $U(13)$, we may also assume that $x\in U(13)^3$. Suppose we have that $y\in U(13)^3$. Then $(x,y)$ is a $U(13)^3$-weighted zero-sum subsequence of $S$ as $yx-xy=0$ and $y,-x\in U(13)^3$.

Suppose $y\in B=U(13)\setminus U(13)^3$. As $U(13)^3=\{\pm 1,\pm 5\}$, we get that $B=\{\pm 2,\pm 3,\pm 4,\pm 6\}$. We can check that $B\subseteq U(13)^3+U(13)^3$. So by Theorem \ref{ls2} we have $U(13)^3x+U(13)^3y=U(13)$ as $x\in U(13)^3$. Thus, there exist $a,b\in U(13)^3$ such that $-z=ax+by$. So $S$ is a $U(13)^3$-weighted zero-sum sequence. Hence, we get that $C_{U(13)^3}(13)\leq 3$. 
\end{proof}

\begin{lem}\label{c7}
We have $D_{U(7)^3}(7)=3$ and $C_{U(7)^3}(7)=4$.
\end{lem}

\begin{proof}
We observe that $U(7)^3=\{\pm 1\}$. 
Let $S=(x,y,z)$ be a sequence in $\mathbb Z_7$ of length 3. We want to show that $S$ has a $\{\pm 1\}$-weighted zero-sum subsequence. We can assume that $x,y,z$ are non-zero and so are in $U(7)=\{\,\pm 1, \pm 2, \pm 3\,\}$. If any two terms of $S$ are equal upto sign, we get a $\{\pm 1\}$-weighted zero-sum subsequence of $S$. Otherwise upto sign and upto a permutation of the terms, the sequence will be $(1,2,3)$ which is a $\{\pm 1\}$-weighted zero-sum sequence. It follows that $D_{\{\pm 1\}}(7)\leq 3$ and so from Theorem \ref{cubp13} we have that $D_{\{\pm 1\}}(7)=3$.

As the sequence $(1,3,1)$ in $\mathbb Z_7$ does not have any $\{\pm 1\}$-weighted zero-sum subsequence of consecutive terms, it follows that $C_{\{\pm 1\}}(7)\geq 4$. Let $S=(x,y,z,w)$ be a sequence in $\mathbb Z_7$. Consider the sequence in $\mathbb Z_7$ having length eight defined as $(x+y,x-y,-x+y,-x-y,z+w,z-w,-z+w,-z-w)$. As at least two terms of this sequence are equal, we get a $\{\pm 1\}$-weighted zero-sum subsequence of consecutive terms of $S$. Thus $C_{\{\pm 1\}}(7)\leq 4$.  
\end{proof}

\section{When $A=U(n)$}

\begin{lem}\label{multi}
Let $n=m_1m_2$ and $A,A_1,A_2$ be subsets of $\mathbb Z_n,\mathbb Z_{m_1},\mathbb Z_{m_2}$. Suppose $f_{n,m_1}(A)\subseteq A_1$ and $f_{n,m_2}(A)\subseteq A_2$. Then $C_A(n)\geq C_{A_1}(m_1)C_{A_2}(m_2)$. 
\end{lem} 

\begin{proof}
Let $C_{A_1}(m_1)=k$ and $C_{A_2}(m_2)=l$. Assume that $k$ and $l$ are at least $2$. There exists a sequence $S_1'=(x_1',\ldots,x_{k-1}')$ of length $k-1$ in $\mathbb Z_{m_1}$ which has no $A_1$-weighted zero-sum subsequence of consecutive terms. Also, there exists a sequence $S_2'=(y_1',\ldots, y_{l-1}')$ of length $l-1$ in $\mathbb Z_{m_2}$ which has no $A_2$-weighted zero-sum subsequence of consecutive terms. 

For each $i\in [1,k-1]$ let $f_{n,m_1}(x_i)=x_i'$ and $S_1=(m_2x_1,\,\ldots,\,m_2x_{k-1})$. For each $j\in [1,l-1]$ let  $f_{n,m_2}(y_j)=y_j'$ and $S_2=(y_1,\ldots,y_{l-1})$. Define a sequence $S$ of length $(k-1)l+l-1=kl-1$ in $\mathbb Z_n$ as  
$$(m_2x_1,\,\ldots,\,m_2x_{k-1},\,y_1,\,m_2x_1,\,\ldots,\,m_2x_{k-1},\,y_2,\,\ldots,\,y_{l-1},\,m_2x_1,\,\ldots,\,m_2x_{k-1})$$ 

Suppose $S$ has an $A$-weighted zero-sum subsequence $T$ of consecutive terms. If $T$ contains some term of $S_2$, we will get a subsequence $S_3$ which has consecutive terms of $S_2$ such that the image of the sequence $S_3$ under $f_{n,m_2}$ has an $A_2$-weighted zero-sum subsequence of consecutive terms, as $f_{n,m_2}(A)\subseteq A_2$. This is not possible by our choice of $S_2'$. Thus, $T$ does not contain any term of $S_2$ and so $T$ is a subsequence of $S_1$. 

Let $T'$ be the sequence in $\mathbb Z_{m_1}$ whose terms are obtained by dividing the terms of $T$ by $m_2$ and taking their images under $f_{n,m_1}$. As $f_{n,m_1}(A)\subseteq A_1$, we will get the contradiction that $T'$ is an $A_1$-weighted zero-sum subsequence of consecutive terms of $S_1'$. Hence, we see that $S$ does not have any $A$-weighted zero-sum subsequence of consecutive terms. As $S$ has length $kl-1$, it follows that $C_A(n)\geq kl$ when both $k$ and $l$ are at least two.

If $k=l=1$, we are done. Suppose exactly one of $k$ and $l$ is equal to one. We can assume that $l=1$ and $k>1$. As the sequence $S_1$ which was defined earlier in this proof does not have any $A$-weighted zero-sum subsequence of consecutive terms, we see that $C_A(n)\geq k$. 
\end{proof}

\begin{cor}\label{unit}
For any natural number $n$, we have $C_{U(n)}(n)\geq 2^{\Omega(n)}$.
\end{cor}

\begin{proof}
If $m$ is a divisor of $n$, we have $f_{n,m}\big(U(n)\big)\subseteq U(m)$. Also, for a prime $p$ we have $U(p)=\mathbb Z_p\setminus \{0\}$. Thus, the result follows from Theorem \ref{nz} and Lemma \ref{multi} by induction on $\Omega(n)$.
\end{proof}

\begin{cor}\label{pm}
Let $n=2^k$ for some $k$. Then $C_{U(n)}(n)=C_{\{\pm 1\}}(n)=n$. 
\end{cor}

\begin{proof}
As $\{1\}\subseteq \{\pm 1\}\subseteq U(n)$, it follows that $C_{U(n)}(n)\leq C_{\{\pm 1\}}(n)\leq C(n)$. Thus, from Corollaries \ref{cn} and \ref{unit} we get $2^k\leq C_{U(n)}(n)\leq C_{\{\pm 1\}}(n)\leq n$. So we see that $C_{U(n)}(n)=C_{\{\pm 1\}}(n)=n$. 
\end{proof}
 
Let $p$ be a prime divisor of $n$. We use the notation $v_p(n)=r$ to mean $p^r\mid n$ and $p^{r+1}\nmid n$. Let $S$ be a sequence in $\mathbb Z_n$. Suppose $p$ is a prime divisor of $n$ with $v_p(n)=r$. Let $S^{(p)}$ be the sequence in $\mathbb Z_{p^r}$ which is the image of the sequence $S$ under $f_{n,p^r}$. The following result is  Observation 2.2 of \cite{sg}. We restate it here using our notation.

\begin{obs}\label{obs}
A sequence $S$ is a $U(n)$-weighted zero-sum sequence in $\mathbb Z_n$ if and only if for every prime divisor $p$ of $n$ we have that the sequence $S^{(p)}$ is a $U\big(p^{v_p(n)}\big)$-weighted zero-sum sequence in $\mathbb Z_{p^{v_p(n)}}$.
\end{obs}

We have the following more general result. For a  prime divisor $p$ of $n$ with $v_p(n)=r$, let $A^j_p=\{x^j:x\in U(p^{r})\}$.

\begin{obs}\label{obs2}
A sequence $S$ is a $U(n)^j$-weighted zero-sum sequence in $\mathbb Z_n$ if and only if for every prime divisor $p$ of $n$ we have that the sequence $S^{(p)}$ is an $A_p^j$-weighted zero-sum sequence in $\mathbb Z_{p^{v_p(n)}}$.  
\end{obs}

\begin{lem}\label{d}
Let $p$ be a prime divisor of $n$ and $n'=n/p$. Suppose $c'\in U(n')$. Then there exists $c\in U(n)$ such that $f_{n,n'}(c)=c'$.
\end{lem}

\begin{proof}
Let $n'=n/p$. If $p$ does not divide $n'$, by the Chinese remainder theorem we have an isomorphism $\psi:\mathbb Z_n\to\mathbb Z_{n'}\times\mathbb Z_p$. If $c\in U(n)$ such that $\psi(c)=(c',1)$, we have that $f_{n,n'}(c)=c'$.

If $p$ divides $n'$, then $n$ and $n'$ have the same prime factors. As $c'$ is coprime to $n'$, it follows that $c'$ is also coprime to $n$. Thus there exists $c'\in U(n)$ such that $f_{n,n'}(c')=c'$.
\end{proof}

\begin{lem}\label{b}
Let $S$ be a sequence in $\mathbb Z_n$ and $p$ be a prime divisor of $n$ which divides every element of $S$. Suppose $n'=n/p$ and $S'$ is the sequence in $\mathbb Z_{n'}$ whose terms are obtained by dividing the terms of $S$ by $p$. If $S'$ is a $U(n')$-weighted zero-sum sequence, then $S$ is a $U(n)$-weighted zero-sum sequence. Also, if $S'$ is a $U(n')^2$-weighted zero-sum sequence, then $S$ is a $U(n)^2$-weighted zero-sum sequence. 
\end{lem}

\begin{proof}
Let $S=(x_1,\ldots,x_k)$. Then $S'=(x_1',\ldots,x_k')$ where for each $i\in [1,k]$ we have   $x_i'=f_{n,n'}(x_i/p)$. 
Suppose $S'$ is a $U(n')$-weighted zero-sum sequence. 
Then for each $i\in [1,k]$ there exist $a_i'\in U(n')$ such that $a_1'x_1'+\cdots+a_k'x_k'=0$. From Lemma \ref{d} we see that for $i\in [1,k]$, there exist $a_i\in U(n)$ such that $f_{n,n'}(a_i)=a_i'$. As $a_1'x_1'+\cdots+a_k'x_k'=0$ it follows that $f_{n,n'}\big((a_1x_1+\cdots+a_kx_k)/p\big)=0$. As $n'$ divides $(a_1x_1+\cdots+a_kx_k)/p$, we get that $n$ divides $a_1x_1+\cdots+a_kx_k$ and so $a_1x_1+\cdots+a_kx_k=0$ in $\mathbb Z_n$.  Thus, $S$ is a $U(n)$-weighted zero-sum sequence. The other assertion can be proved in a  similar manner. 
\end{proof}

For the next theorem we need the following (\cite{sg}, Lemma 2.1 (ii)), which we restate here using our terminology:

\begin{lem}\label{gri}
Let $p$ be an odd prime. If a sequence $S$ over $\mathbb Z_{p^r}$ has at least two terms coprime to $p$, then $S$ is a $U(p^r)$-weighted zero-sum sequence. 
\end{lem}

\begin{thm}\label{unit2}
When $n$ is odd, we have $C_{U(n)}(n)\leq  2^{\Omega(n)}$.
\end{thm}

\begin{proof}
We prove this theorem by induction on $\Omega(n)$. Let $S=(x_1,\ldots,x_k)$ be a sequence in $\mathbb Z_n$ of length $k=2^{\Omega(n)}$. If $\Omega(n)=1$ then $n$ is prime and so $U(n)=\mathbb Z_n\setminus\{0\}$. Hence we are done by using Theorem \ref{nz}. Let us now assume that $\Omega(n)>1$.

\begin{case}
For any prime divisor $p$ of $n$ at least two terms of $S$ are coprime to $p$. 
\end{case}

Let $p$ be a prime divisor of $n$ and let $r=v_p(n)$. Let $S^{(p)}$ be as defined before Observation \ref{obs}. Then $S^{(p)}$ has at least two units. As $n$ is odd it follows that $p$ is an odd prime. Hence by Lemma \ref{gri} we see that $S^{(p)}$ is a $U(p^r)$-weighted zero-sum sequence in $\mathbb Z_{p^r}$. As this is true for any prime divisor $p$ of $n$, by Observation \ref{obs} we see that $S$ is a $U(n)$-weighted zero-sum sequence.

\begin{case}
There is a prime divisor $p$ of $n$ such that at most one term of $S$ is coprime to $p$. 
\end{case}

By partitioning $S$ into two equal halves where each half has $k/2$ consecutive terms, we see that there is a subsequence $T$ of consecutive terms of $S$ of length $k/2$ such that $p$ divides every term of $T$.

Let $n'=n/p$ and $T'$ denote the sequence in $\mathbb Z_{n'}$ whose terms are obtained by dividing the terms of $T$ by $p$. As $\Omega(n')=\Omega(n)-1$ and as $T'$ is a sequence of length $2^{\Omega(n')}$ in $\mathbb Z_{n'}$, by the induction hypothesis $T'$ has a $U(n')$-weighted zero-sum subsequence of consecutive terms. From Lemma \ref{b} we see that $T$ has a $U(n)$-weighted zero-sum subsequence of consecutive terms. As $T$ is a subsequence of consecutive terms of $S$, it follows that $S$ has a $U(n)$-weighted zero-sum subsequence of consecutive terms. 
\end{proof}

\begin{cor}\label{A}
When $n$ is odd, we have $C_{U(n)}(n)=2^{\Omega(n)}$.  
\end{cor}

\begin{proof}
This follows from Corollary \ref{unit} and Theorem \ref{unit2}.
\end{proof}

\section{When $A=U(n)^2$}

This section is a generalisation of the results obtained in Section \ref{q}. We begin with the following observation. 

\begin{cor}\label{sq'}
If $n=2^rm$ where $m$ is odd, then $C_{U(n)^2}(n)\geq 2^r3^{\Omega(m)}$.
\end{cor}

\begin{proof}
If $m$ is a divisor of $n$, then $f_{n,m}(U(n)^2)\subseteq U(m)^2$. Also, for a prime $p$ we have  $U(p)^2=Q_p$. Thus, the result follows from Theorem \ref{q1} and Lemma \ref{multi} by induction on $\Omega(n)$. 
\end{proof}

For the next theorem we need the following result which follows immediately from (\cite{CM}, Lemma 1). 

\begin{lem}\label{cm}
Let $n=p^r$ where $p$ is a prime which is at least seven. Suppose $S$ is a sequence in $\mathbb Z_n$ such that at least three terms of $S$ are in $U(n)$. Then $S$ is a $U(n)^2$-weighted zero-sum sequence. 
\end{lem}

\begin{rem}
The conclusion of Lemma \ref{cm} may not hold when $p<7$. When $n$ is 2 or 5, the sequence $(1,1,1)$ in $\mathbb Z_n$ is not a $U(n)^2$-weighted zero-sum sequence. The sequence $(1,2,1)$ in $\mathbb Z_3$ is not a $U(3)^2$-weighted zero-sum sequence. 
\end{rem}

\begin{thm}\label{sq3}
If every prime divisor of $n$ is at least $7$, then $C_{U(n)^2}(n)\leq 3^{\Omega(n)}$.  
\end{thm}

\begin{proof}
Let $S=(x_1,\ldots,x_k)$ be a sequence in $\mathbb Z_n$ of length $k=3^{\Omega(n)}$. We want to show that $S$ has a $U(n^2)$-weighted zero-sum subsequence of consecutive terms. We will use induction on $\Omega(n)$. From Theorem \ref{q1} we see that for any prime $p\geq 3$ we have $C_{Q_p}(p)= 3$. 
Let us now assume that $\Omega(n)>1$.

\begin{case}
For any prime divisor $p$ of $n$ at least three terms of $S$ are coprime to $p$. 
\end{case}

Let $p$ be a prime divisor of $n$ and $v_p(n)=r$. Then $S^{(p)}$ has at least three units, where $S^{(p)}$ is as defined before Observation \ref{obs}. As $p$ is at least 7, by Lemma \ref{cm} we see that  $S^{(p)}$ is a $U(p^r)^2$-weighted zero-sum sequence in $\mathbb Z_{p^r}$. As this is true for every prime divisor of $n$, by Observation \ref{obs2} it follows that $S$ is a $U(n)^2$-weighted zero-sum sequence.

\begin{case}
There is a prime divisor $p$ of $n$ such that at most two terms of $S$ are coprime to $p$. 
\end{case}

By partitioning $S$ into three equal parts where each part has $k/3$ consecutive terms, we see that there is a subsequence $T$ of consecutive terms of $S$ of length $k/3$ such that $p$ divides every term of $T$. Let $n'=n/p$ and $T'$ denote the sequence in $\mathbb Z_{n'}$ whose terms are obtained by dividing the terms of $T$ by $p$. As $\Omega(n')=\Omega(n)-1$ and $T'$ is a sequence of length $3^{\Omega(n')}$ in $\mathbb Z_{n'}$, by the induction hypothesis it follows that $T'$ has a $U(n')^2$-weighted zero-sum subsequence of consecutive terms. By Lemma \ref{b} we see that $T$ has a $U(n)^2$-weighted zero-sum subsequence of consecutive terms. As $T$ is a subsequence of consecutive terms of $S$, it follows that $S$ has a $U(n)^2$-weighted zero-sum subsequence of consecutive terms. 
\end{proof}

\begin{cor}\label{B}
If every prime divisor of $n$ is at least 7, then $C_{U(n)^2}(n)=3^{\Omega(n)}$.
\end{cor}

\begin{proof}
This follows from Corollary \ref{sq'} and Theorem \ref{sq3}.
\end{proof}

\section{When $A=U(n)^3$}\label{cubes} 

Let $n=p_1^{r_1}p_2^{r_2}\ldots p_s^{r_s}$ where the $p_i$'s are distinct primes and $7\nmid n$.   
Consider the set $I=\{\,i:p_i\equiv 1~(mod~3)\,\}$. Let $n_1=\prod\,\{p_i^{r_i}:i\in I\}$ and let $n_2=n/n_1$. We will follow the notation  \framebox{  $n=n_1n_2$} throughout this section.

\begin{cor}\label{cub2'}
Let $m=7^rn$ where $7\nmid n$ and $n=n_1n_2$ where $n_1,n_2$ are as defined above. Then $C_{U(m)^3}(m)\geq 2^{\Omega(n_2)}3^{\Omega(n_1)}4^r$.
\end{cor}

\begin{proof}
In Theorem \ref{cubp13} we have seen that $C_{U(p)^3}(p)=3$ when $p\equiv 1~(mod~3)$ and $p\neq 7$. In Lemma \ref{c7} we have seen that $C_{U(7)^3}(7)=4$. If $p\not\equiv 1~(mod~3)$, then $U(p)^3=U(p)=\mathbb Z_p\setminus\{0\}$ and so by Theorem \ref{nz} we have $C_{U(p)^3}(p)=2$. Thus, this result now follows from Lemma \ref{multi} by induction on $\Omega(n)$ and by using the fact that if $d$ is a divisor of $n$, then $f_{n,d}(U(n)^3)\subseteq U(d)^3$.
\end{proof}

\begin{thm}\label{cub4'}
Let $n$ be a squarefree number which is not divisible by 2, 7 or 13. Then $C_{U(n)^3}(n)\leq 2^{\Omega(n_2)}3^{\Omega(n_1)}$. 
\end{thm}

\begin{proof}
Let $S$ be a sequence in $\mathbb Z_n$ of length $k=2^{\Omega(n_2)}3^{\Omega(n_1)}$. We want to show that $S$ has a $U(n)^3$-weighted zero-sum subsequence of consecutive terms. We now prove this theorem by induction on $\Omega(n)$.

Suppose $n=p$ where $p$ is a prime. If $p\equiv 1~(mod~3)$ then $n=n_1$, and by using Lemma \ref{cub5} we can show that $C_{U(n)^3}(n)\leq 3$. If $p\not\equiv 1~(mod~3)$ then $n=n_2$, and by using Lemma \ref{gri} we can show that $C_{U(n)^3}(n)\leq 2$. Let us now assume that $\Omega(n)>1$.

\begin{case}
For any prime divisor $p$ of $n_1$ at least three terms of $S$ are coprime to $p$, and for any prime divisor $p$ of $n_2$ at least two terms of $S$ are coprime to $p$. 
\end{case}
 
Let $p$ be a prime divisor of $n$ and $S^{(p)}$ be as defined before Observation \ref{obs}. If $p$ divides $n_1$ then $S^{(p)}$ has at least three units. So by Lemma \ref{cub5} we get that $S^{(p)}$ is a $U(p)^3$-weighted zero-sum sequence in $\mathbb Z_p$ as $p\neq 7,13$. If $p$ divides $n_2$ then $S^{(p)}$ has at least two units. So by Lemma \ref{gri} we get that $S^{(p)}$ is a $U(p)$-weighted zero-sum sequence in $\mathbb Z_p$ as $p\neq 2$.

We have seen that when $p\not \equiv 1~(mod~3)$ we have $U(p)^3=U(p)$. Thus, for every prime divisor $p$ of $n$ we get that $S^{(p)}$ is a $U(p)^3$-weighted zero-sum sequence in $\mathbb Z_p$ and so by Observation \ref{obs2} we see that $S$ is a $U(n)^3$-weighted zero-sum sequence. 

\begin{case}
There is a prime divisor $p$ of $n_1$ such that at most two terms of $S$ are coprime to $p$. 
\end{case}

Let $n'=n/p$. If we write $n'$ as $n_1'n_2'$ as per the notation given at the beginning of this section, it follows that $n_1'=n_1/p$ and $n_2'=n_2$. By partitioning $S$ into three equal parts where each part has $k/3$ consecutive terms, we see that there is a subsequence $T$ which has consecutive terms of $S$ and length $k/3$ such that $p$ divides every term of $T$. Let $T'$ denote the sequence in $\mathbb Z_{n'}$ whose terms are obtained by dividing the terms of $T$ by $p$.

Then $T'$ is a sequence of length $k/3=2^{\Omega(n_2)}3^{\Omega(n_1)-1}=2^{\Omega(n'_2)}3^{\Omega(n'_1)}$ in $\mathbb Z_{n'}$. As $n'$ is squarefree and is not divisible by 2, 7 or 13 and as $\Omega(n')=\Omega(n)-1$, by the induction hypothesis we see that $T'$ has a $U(n')^3$-weighted zero-sum subsequence of consecutive terms. By a similar argument as in Lemma \ref{b} (where we replace squares by cubes), we see that $T$ has a $U(n)^3$-weighted zero-sum subsequence of consecutive terms. As $T$ is a subsequence of consecutive terms of $S$, it follows that $S$ has a $U(n)^3$-weighted zero-sum subsequence of consecutive terms.

\begin{case}
There is a prime divisor $p$ of $n_2$ such that at most one term of $S$ is coprime to $p$. 
\end{case}

Let $n'=n/p$. It follows that $n_1'=n_1$ and $n_2'=n_2/p$. By partitioning $S$ into two equal parts where each part has $k/2$ consecutive terms, we see that there is a subsequence $T$ which has  consecutive terms of $S$ and length $k/2$ such that $p$ divides every term of $T$. Let $T'$ denote the sequence in $\mathbb Z_{n'}$ whose terms are obtained by dividing the terms of $T$ by $p$.

Then $T'$ is a sequence of length $k/2=2^{\Omega(n_2)-1}3^{\Omega(n_1)}=2^{\Omega(n'_2)}3^{\Omega(n'_1)}$ in $\mathbb Z_{n'}$. By a similar argument as in the previous case we see that $S$ has a $U(n)^3$-weighted zero-sum subsequence of consecutive terms. 
\end{proof}

\begin{cor}\label{cub4''}
Let $n$ be a squarefree number which is not divisible by 2, 7 or 13. Then $C_{U(n)^3}(n)=2^{\Omega(n_2)}3^{\Omega(n_1)}$. 
\end{cor}

\begin{proof}
This follows from Corollary \ref{cub2'} and Theorem \ref{cub4'}.  
\end{proof}

\begin{rem}
By Corollary \ref{cub2'} we see that the conclusions of Theorem \ref{cub4'} and Corollary \ref{cub4''} are not true when $n$ is divisible by 7. 
\end{rem}

We will now give an upper bound for $C_{U(n)^3}(n)$ when $n$ is not squarefree, for which we need the following (\cite{ss}, Lemma 4).

\begin{lem}\label{cub3}
Let $n=p^r$ where $p$ is a prime such that $p\geq 13$ and $p\equiv 1~(mod~3)$. Let $S$ be a sequence in $\mathbb Z_n$ such that at least four terms of $S$ are units. Then $S$ is a $U(n)^3$-weighted zero-sum sequence.   
\end{lem}

\begin{rem}
Let $p$ be a prime and $r\geq 2$. As $|U(p^r)|=p^{r-1}(p-1)$, we see that if $p\equiv 2~(mod~3)$ then $3$ does not divide $|U(p^r)|$. So the homomorphism from $U(p^r)\to U(p^r)$ given by $x\mapsto x^3$ has trivial kernel and hence it is onto. Thus we have $U(p^r)^3=U(p^r)$.
\end{rem}

\begin{cor}\label{cub3'}
Let $n=p^r$ where $p$ is an odd prime such that $p\neq 3,7$. Let $S$ be a sequence in $\mathbb Z_n$ such that at least four elements of $S$ are units. Then $S$ is a $U(n)^3$-weighted zero-sum sequence.   
\end{cor}

\begin{proof}
When $p\equiv 2~(mod~3)$, the result follows from Lemma \ref{gri} as $p$ is odd and we have $U(n)^3=U(n)$. When $p\equiv 1~(mod~3)$, the result follows from Lemma \ref{cub3} as $p\neq 7$.
\end{proof}

\begin{rem}
The conclusion of Corollary \ref{cub3'} is false when $p=2,3,7$. 

As $U(7)^3=\{1,-1\}$, the sequence $(1,1,1,1,1)$ in $\mathbb Z_7$ is not a $U(7)^3$-weighted zero-sum sequence.

As $U(9)^3=\{1,-1\}$, the sequence $(1,1,1,1,1)$ in $\mathbb Z_9$ is not a $U(9)^3$-weighted zero-sum sequence. 

As the sequence $(1,1,1,1,1)$ in $\mathbb Z_2$ is not a zero-sum sequence and as the image of $U(2^r)^3$ under $f_{2^r,2}$ is $\{1\}$, it follows that the  sequence $(1,1,1,1,1)$ in $\mathbb Z_{2^r}$ is not a $U(2^r)^3$-weighted zero-sum sequence.
\end{rem}

\begin{thm}\label{cub4}
If $n$ is not divisible by 2, 3 or 7, then $C_{U(n)^3}(n)\leq 2^{\Omega(n_2)}4^{\Omega(n_1)}$. 
\end{thm}

\begin{proof}
Let $S$ be a sequence in $\mathbb Z_n$ of length $k=2^{\Omega(n_2)}4^{\Omega(n_1)}$. We want to show that $S$ has a $U(n)^3$-weighted zero-sum subsequence of consecutive terms. We now prove this theorem by induction on $\Omega(n)$. If $n$ is a prime, we use a similar argument as in the first paragraph of the proof of Theorem \ref{cub4'}. Let us now assume that $\Omega(n)>1$.

\begin{case}
For any prime divisor $p$ of $n_1$ at least four terms of $S$ are coprime to $p$, and for any prime divisor $p$ of $n_2$ at least two terms of $S$ are coprime to $p$.
\end{case}

Let $p$ be a prime divisor of $n$, $v_p(n)=r$ and $S^{(p)}$ be as defined before Observation \ref{obs}. If $p$ divides $n_1$, then $S^{(p)}$ has at least four units. So from Lemma \ref{cub3} we get that $S^{(p)}$ is a $U(p^r)^3$-weighted zero-sum sequence in $\mathbb Z_{p^r}$ as $p\neq 7$. If $p$ divides $n_2$, then $S^{(p)}$ has at least two units. So by Lemma \ref{gri} we get that $S^{(p)}$ is a $U(p^r)$-weighted zero-sum sequence in $\mathbb Z_{p^r}$ as $p\neq 2$. 

We have seen that when $p\equiv 2~(mod~3)$ we have  $U(p^r)^3=U(p^r)$. As $3\nmid n$, for every prime divisor $p$ of $n$ we get that $S^{(p)}$ is a $U(p^r)^3$-weighted zero-sum sequence in $\mathbb Z_{p^r}$. So by Observation \ref{obs2} we see that $S$ is a $U(p^r)^3 $-weighted zero-sum sequence. 

\begin{case}
There is a prime divisor $p$ of $n_1$ such that at most three terms of $S$ are coprime to $p$, or there is a prime divisor $p$ of $n_2$ such that at most one term of $S$ is coprime to $p$.
\end{case}

The proof of the result in this case is very similar to the proofs of the corresponding cases in Theorem \ref{cub4'}.
\end{proof}

\section{Concluding remarks}

In Corollary \ref{A} we have determined $C_A(n)$ for $A=U(n)$ when $n$ is odd. The corresponding result for an even integer can be investigated. It will also be interesting to see whether the lower bounds in Corollaries \ref{sq'} and \ref{cub2'} are the values of $C_A(n)$ for $A=U(n)^2$ and $A=U(n)^3$ respectively.

\bigskip

{\bf Acknowledgement.}
Santanu Mondal would like to acknowledge CSIR, Govt of India for a research fellowship. We would like to thank Dr. Subha Sarkar and Ms Shruti Hegde from RKMVERI for helpful discussions. We are grateful to the referees whose suggestions were helpful in improving the presentation of the paper.

\end{document}